\documentclass[letterpaper,reqno,12pt]{amsart}
\usepackage[margin=1.2in]{geometry}

\usepackage{amsmath} \usepackage{amssymb} \usepackage{amsthm}
\usepackage{amscd} 
\usepackage[all]{xy} 
\usepackage{enumerate}
\usepackage{mathrsfs}
\usepackage[dvipdfm]{hyperref}
\usepackage{lipsum}
\hypersetup{colorlinks=true}

\def\phi{\varphi}
\def\epsilon{\varepsilon}
\def\tilde{\widetilde}
\def\bar{\overline}
\def\to{\longrightarrow}
\def\mapsto{\longmapsto}

\def\Spec{\operatorname{Spec}}
\def\Proj{\operatorname{Proj}}
\def\Supp{\operatorname{Supp}}

\def\ord{\operatorname{ord}}

\newcommand{\Vol}{\mathrm{Vol}(X,0)}
\newcommand{\AX}{\mathrm{A}_{\X/X}}
\newcommand{\AY}{\mathrm{A}_{\Y/Y}}
\newcommand{\env}{\mathrm{Env}}

\newcommand{\Q}{\mathbb{Q}} 
\newcommand{\C}{\mathbb{C}} 
\newcommand{\R}{\mathbb{R}}

\newcommand{\sO}{\mathcal{O}}
\newcommand{\fa}{\mathfrak{a}}

\newcommand{\m}{\mathfrak{m}}
\newcommand{\n}{\mathfrak{n}}

\newcommand{\X}{\mathfrak{X}}
\newcommand{\Y}{\mathfrak{Y}}

\newcommand{\ilim}[1][]{\mathop{\varinjlim}\limits_{#1}}
\newcommand{\plim}[1][]{\mathop{\varprojlim}\limits_{#1}}

\theoremstyle{plain}
\newtheorem{thm}{Theorem}[section] 
\newtheorem{cor}[thm]{Corollary}
\newtheorem{prop}[thm]{Proposition}
\newtheorem{conj}[thm]{Conjecture}

\newtheorem{lem}[thm]{Lemma}
\theoremstyle{definition} 
\newtheorem{defn}[thm]{Definition}
\newtheorem{conv}[thm]{Convention}
\newtheorem{propdef}[thm]{Proposition-Definition} 
\newtheorem{eg}[thm]{Example} 

\theoremstyle{remark}
\newtheorem{rem}[thm]{Remark}

\newtheorem*{cl}{Claim}
\newtheorem*{clproof}{Proof of Claim}

\newtheorem*{acknowledgement}{Acknowledgments}

\title{Singularities of non-$\Q$-Gorenstein varieties admitting a polarized endomorphism}
\author{Shou Yoshikawa}
\address{Graduate School of Mathematical Sciences, University of Tokyo, 3-8-1 Komaba, Meguro-ku, Tokyo 153-8914, Japan}
\email{yoshikaw@ms.u-tokyo.ac.jp}

\begin{document}
\tolerance = 9999

\maketitle
\markboth{SHOU YOSHIKAWA}{Singularities of non-$\Q$-Gorenstein varieties admitting a polarized endomorphism}

\begin{abstract}
In this paper, we discuss a generalization of log canonical singularities in the non-$\Q$-Gorenstein setting.
We prove that if a normal complex projective variety has a non-invertible polarized endomorphism,
then it has log canonical singularities in our sense.
As a corollary, we give an affirmative answer to a conjecture of \cite{BH}.
\end{abstract}

\section{Introduction}

Let $X$ be a normal complex projective variety admitting a non-invertible polarized endomorphism.
Broustet and H\"{o}ring showed in \cite{BH} that
if $X$ is $\Q$-Gorenstein, then $X$ has log canonical singularities,
 using the existence of the log canonical model of $X$.
They also conjectured that an analogous statement holds even if $X$ is not $\Q$-Gorenstein.

\begin{conj}[\textup{cf. \cite[Conjecture 1.5]{BH}}]
Let $X$ be a normal projective complex variety admitting a non-invertible polarized endomorphism 
$f : X \to X$.
Suppose that $X$ has the log canonical model $\mu : Y \to X$.
Then $\mu$ is an isomorphism in codimension one.
\end{conj}

In this paper, we give an affirmative answer to Conjecture 1.1.
For this purpose, we consider a generalization of log canonical singularities for non-$\Q$-Gorenstein varieties.
Classical log discrepancies are defined in terms of the pullback of canonical divisors,
but the classical pullback of a divisor makes sense only when the divisor is $\Q$-Cartier\footnote{In the surface case, we can consider Mumford's numerical pullback.
In the  higher dimensional case, de Fernex and Hacon defined the pullback of non-$\Q$-Cartier divisors in \cite{dFH}.}.
Hence we use the notion of nef envelopes introduced by \cite{BdFF} instead of the classical pullbacks.

When $D$ is a Weil divisor on a normal complex variety $W$,
the nef envelope $\env_W(D)$ of $D$ is the compatible system $ \{ \env_W(D)_Y \}$ 
with respect to the push-out map, where $\env_W(D)_Y$ is a divisor on a birational model $Y$ of $W$.
$\env_W(D)$ satisfies the following properties.

\begin{itemize}
\item  $\ord_E (\env_W(D)) =\displaystyle{ \lim_{m \to \infty}} \frac{1}{m} \ord_E (\sO_W (mD)) $ for any prime divisor $E$ over $W$.
\item If $D$ is $\Q$-Cartier, then $\env_W(D)_Y = \pi^* D$ for any birational model $\pi : Y \to W$.
\item If $W$ is a surface, then the nef envelope coincides with the Mumford's numerical pullback.
\end{itemize}
We define the log discrepancies of $W$ using $\env_W(K_W)$,
and we say that $W$ has \emph{valuative log canonical singularities} if  the minimal log discrepancy is greater than or equal to 0.

Thanks to the following theorem, we can reduce Conjecture 1.1 to prove that $X$ has valuative log canonical singularities.

\begin{thm}[\textup{Theorem 4.4}]
The following are equivalent to each other. 
\begin{enumerate}
\item
$W$ has valuative log canonical singularities.

\item
For any birational model $W_\pi$ over $W$ and for any positive number $m$,
we have
\[
\pi_* \sO_{W_\pi} (m(K_{W_\pi} + E^\pi)) = \sO_W(mK_W)
\]
where $E^\pi$ is the sum of the exceptional prime divisors on $W_\pi$.

\end{enumerate}
Furthermore, if $W$ has the log canonical model, the following condition is also equivalent.

(3) The log canonical model of $W$ is an isomorphism in codimension one.

\end{thm}

Next, we discuss a local version of the main theorem.

\begin{thm}[\textup{Theorem 4.11}]
Suppose that $ (R,\m,k) $ is a normal local ring essentially of finite type over $\C$
and $R$ has a finite injective local homomorphism $\phi : R \to R$.
If $\Spec R \backslash \{ \m \}$ has valuative log canonical singularities
and $\deg (\phi) > [\phi_*k : k]$,
then $\Spec R$ has valuative log canonical singularities.
\end{thm}

Assume that the non-valuative log canonical locus of $X$ is not empty, and take an irreducible component $Z$.
Since $Z$ is totally invariant up to replacing $f$ by some iterate,
$f$ induces an endomorphism of the local ring $\sO_{X,\eta}$ at the generic point $\eta$ of $Z$.
Applying Theorem 1.3 to $\sO_{X,\eta}$, we have $\deg (f) = [f_* \kappa (Z) : \kappa (Z)] $,
where $\kappa (Z)$ is the residue field of $Z$.
Since $[f_* \kappa (Z) : \kappa (Z)] $ is equal to $\deg (f |_Z) $,
we see that $\deg (f) = \deg (f |_Z)$, but it contradicts the fact that $f$ is a non-invertible polarized endomorphism.
Thus, we have the following main theorem, which gives an affirmative answer to Conjecture 1.1.

\begin{thm}[\textup{Theorem 5.2}]
Let $X$ be a normal complex projective variety admitting a non-invertible polarized endomorphism.
Then $X$ has valuative log canonical singularities.
\end{thm}

When $X$ is $\Q$-Gorenstein, Theorem 1.4 is nothing but \cite[Theorem 1.3]{BH},
but since we do not use the existence of the log canonical model,
our proof is different from the proof given in \cite{BH}.

We also discuss another important conjecture on polarized endomorphisms.

\begin{conj}[\textup{\cite[Conjecture1.2]{BG}}]
Let $X$ be a normal complex projective variety admitting a non-invertible polarized endomorphism.
Then $X$ is of Calabi-Yau type.
\end{conj}

Here we say that $X$ is of \emph{Calabi-Yau type} if
there exists an effective $\Q$-Weil divisor $\Delta$ on $X$ such that
$K_X + \Delta $ is $\Q$-linearly trivial and $(X,\Delta)$ is a log canonical pair.
Broustet and Gongyo proved in \cite{BG} that
Conjecture 1.5 holds if
$X$ is $\Q$-Gorenstein and $f$ is \'{e}tale in codimension one.
We generalize their result to the case where $X$ is not necessarily $\Q$-Gorenstein.

\begin{thm}[\textup{Theorem 6.10}]
Let $X$ be a normal complex projective variety admitting a non-invertible polarized endomorphism $f$.
Suppose that $f$ is \'{e}tale in codimension one.
Then $K_X$ is $\Q$-linearly trivial and $X$ has log canonical singularities,
 and in particular, $X$ is of Calabi-Yau type.
\end{thm}

To prove this theorem, we make use of the following properties of numerically $\Q$-Cartier divisors introduced by \cite{BdFF}.

\begin{itemize}
\item Let $D$ be a Weil divisor on $Z$, and then $D$ is numerically $\Q$-Cartier 
if and only if $\env_W(D) + \env_W(-D) = 0 $.
\item If $\displaystyle{\bigoplus_m} \sO_W(mD) $ is finitely generated, then 
$D$ is numerically $\Q$-Cartier if and only if $D$ is $\Q$-Cartier.

\end{itemize}

\begin{acknowledgement}
The author wishes to express his gratitude to his supervisor Professor Shunsuke Takagi for his encouragement, valuable advice and suggestions.
He is also grateful to Professor Yoshinori Gongyo, Professor Andreas H\"{o}ring,
Professor S\'{e}bastien Boucksom and Dr. Kenta Hashizume  for their helpful comments and suggestions.
This work was supported by the Program for Leading Graduate Schools, MEXT, Japan.

\end{acknowledgement}

\section{Shokurov's $b$-dividors}
In this section, we recall the definition and some basic properties of b-divisors.
See \cite{BdFF} for the details.

\subsection{Weil $b$-divisors and Cartier $b$-divisors}

\begin{conv}
Throughout this paper,X is a normal integral scheme 
essentially of finite type over $\C$.
If $\pi:X_\pi \to X$ is a projective birational morphism and $X_\pi $ is normal, then we say $\pi$ is a birational model over X.
And if $\pi$ and $\pi' $ are birational models and the birational morphism $X_{\pi'} \to X_{\pi} $  over $X$ exists,
then we denote $\pi' \geq \pi $. 
\end{conv}

\begin{defn}[\textup {cf. \cite[Subsection 1.2]{BdFF}}]
We define the group of \emph{Weil b-divisors} over $X$ as
\[
\mathrm{Div}(\X) :=\plim[\pi]\mathrm{Div}(X_{\pi})
\]
where $ \mathrm{Div} (X_\pi) $ denotes the group of Weil divisors of $X_\pi$ and the limit is taken 
with respect to the push-forward maps  
$ \mathrm{Div}(X_{\pi'}) \rightarrow \mathrm{Div}(X_\pi) $
,which are defined whenever $\pi' \geq \pi $. 

A Weil b-divisor $W$ over $X$ consists of a family of Weil divisors $W_\pi \in \mathrm{Div}(X_\pi)$ 
that are compatible under push-forward. 
We say $W_\pi$ is the \emph{trace} of $W$ on the model $X_{\pi} $.

\end{defn}

\begin{defn}(\textup {cf. \cite[Subsection 1.2]{BdFF}})
The group of \emph{Cartier b-divisors} over $X$ is defined as
\[
\mathrm{CDiv}(\X) := \ilim[\pi]\mathrm{CDiv}(X_\pi)
\]
with $\mathrm{CDiv}(X_\pi)$ denoting the group of Cartier divisors of $X_\pi$.
Here the limit is taken with respect to the pull-back maps 
$\mathrm{CDiv}(X_\pi) \to \mathrm{CDiv}(X_{\pi'})$,
which are defined whenever  $\pi' \geq \pi $. 
We denote $\bar{C_\pi}$ the image of Cartier divisor $C_\pi$ on $X_\pi$ of the map 
$\mathrm{CDiv}(X_\pi) \to \mathrm{CDiv}(\X)$
\end{defn}

There are an injection $\mathrm{CDiv}(\X) \to \mathrm{Div}(\X)$ determined by the cycle maps on   birational models $X_\pi$.

Let $C$ be a Cartier b-divisor over $X$. 
We can find a birational model $\pi$ such that $C = \bar{C_\pi}$.
We call $C$ is determined by a birational model $\pi$.

An element of $\mathrm{CDiv}_\R (\X) := \ilim[\pi]\mathrm{CDiv}_\R (X_\pi)$
 (resp.$\mathrm{CDiv}_\R (\X) := \ilim[\pi]\mathrm{CDiv}_\R (X_\pi)$)
 will be called a $\R$-Weil b-divisor (resp. $\R$-Cartier b-divisor), 
and similarly with $\Q$ in replace of $\R$.

\begin{eg}
\ 
\begin{enumerate}
\item The system of the canonical divisors ($K_{X_\pi}$)$_\pi$ is an inverse system of divisors.
We denote the corresponding b-divisor by $K_{\X} \in \mathrm{Div}(\X)$, and we call this the \emph{canonical $b$-divisor}.

\item Given a coherent fractional ideal $\fa$ on $X$,
 we denote $Z(\fa)$ the Cartier $b$-divisor determined on 
the normalized blow-up $X_\pi$ of $X$ along $\fa$ by
\[
\fa \cdot \sO_{X_\pi} = \sO_{X_\pi} (Z(\fa)_\pi)
\]
i,e, $Z(\fa) = \bar{Z(\fa)_\pi}$. If $\fa$ is contained in $\sO_X$, then $Z(\fa) \leq 0$. 
\end{enumerate}
\end{eg}

\begin{defn}\label{convergence of $b$-divisor}[\textup {cf. \cite[Lemma 1.12]{BdFF}}]
Let $\{W_i\}_{i \in I}$ be a net of $\R$-Weil $b$-divisors and $W$ be an $\R$-Weil $b$-divisor.
Now, $\{W_i\}_{i \in I}$ converges to $W$ if the following conditions are satisfied.
\begin{enumerate}
\item There exists a finite dimensional vector space $V$ in $\mathrm{Div}_{\R} (X)$
such that $W_{i,X} \in V$ for any $i \in I$ and
\item For any prime divisor $E$ over $X$,
$\ord_E (W_i)$ converges to $\ord_E (W)$ in $\R$.
\end{enumerate}
\end{defn}

\begin{eg}
We say $\fa_\bullet = (\fa_m)_{m \geq 0}$ is a \emph{graded sequence of fractional ideal sheaves} if $\fa_0 = \sO_X$, 
each $\fa_m$ is a fractional ideal sheaf of $X$ and $\fa_k \cdot \fa_m \subset \fa_{m+k}$ for every $k,m$.
We say that $\fa_\bullet$ has \emph{linearly bounded denominators} if there exists a Weil divisor 
$D$ on $X$ such that $\sO_X (mD) \cdot \fa_m \subset \sO_X$ for any m. 

If $\fa_\bullet$ is a graded sequence of fractional ideal sheaves with linearly bounded denominators, 
then the sequence of $b$-divisors $ \{ \frac{1}{m} Z(\fa_m) \} $ converges.
We denote
\[
Z(\fa_\bullet) = \lim \frac{1}{m} Z(\fa_m)
\]
and we say $Z(\fa_\bullet)$ is the $b$-divisor associated to a graded sequence of fractional ideals 
$\{ \fa_m \}$.
\end{eg}
\subsection{Nef $b$-divisors and nef envelopes}

In this subsection, we recall the definition of nef $b$-divisors, nef envelopes and the negativity lemma for $b$-divisor(see Proposition 2.12.) which is very important tool to study nef $b$-divisors.

\begin{defn}[\textup{cf. \cite[Definition 2.9]{BdFF}}] 
An $\R$-Cartier $b$-divisor $D$ over $X$ is an \emph{$X$-nef} $\R$-Cartier $b$-divisor if 
$D$ is determined on a birational model $\pi$ and $D_\pi$ is an $X$-nef divisor on $X_\pi$.

An $\R$-Weil $b$-divisor $W$ over $X$  is \emph{$X$-nef} $\R$-Weil $b$-divisor if
there exists a net of $X$-nef Cartier $b$-divisors $\{W_i \}_i $ such that 
$\{W_i \}_i $ converges to $W$.
\end{defn}

\begin{defn}[\textup{\cite[Definition 2.3]{BdFF}}]
Let $D$ be an $\R$-Weil divisor on a birational model $X_\pi$ over $X$.
The \emph{nef envelope} $\env_\pi (D)$ of $D$ over $X$ is the $b$-divisor associated to
$\{ \pi_* \sO_{X_\pi} (mD) \}$.
If $\pi=\mathrm{id}$, we denote $\env_\pi (D)$ by $\env_X (D)$.
\end{defn}

\begin{eg}
\  
\begin{enumerate}
\item $Z(\fa)$ in example 2.4.(1) is an $X$-nef $\R$-Cartier $b$-divisor over $X$.
Therefore $Z(\fa_\bullet)$ is an $X$-nef $\R$-Weil $b$-divisor over $X$ and 
in particular, nef envelopes are $X$-nef $\R$-Weil $b$-divisors.
\item If $D$ is an $\R$-Cartier divisor on $X$, then $\env_X (D) = \bar {D}$. 
If $D$ is an $\R$-Weil divisor on $X$, then $\env_X (D)_X = D$, but in general, 
the divisor $\env_X (D)$ is not equal to  $-\env_X (-D) $..
\end{enumerate}
\end{eg}

\begin{prop}
Let $D$ and $D'$ be $\R$-Weil divisors on $X_\pi$.Then 
\begin{enumerate}
\item $\env_\pi (D+D') \geq \env_\pi (D) + \env_\pi (D')$ and
\item $\env_\pi(tD) = t \env_\pi (D) $ for any $t \in \R_{\geq 0}$.
\end{enumerate}
\end{prop}

\begin{proof}
For each $m \geq 0$ we have
\[
(\pi_* \sO_{X_\pi} (mD) ) \cdot (\pi_* \sO_{X_\pi} (mD') ) \subset \pi_* \sO_{X_\pi} (m(D+D'))
\]
hence we get the first statement.

We may assume $D$ is an effective divisor.
If $t$ is a rational number, then the result is followed directly from the definition of nef envelopes.
In general case, 
\[
t_j \env_\pi (D) = \env_\pi (t_j D) \geq \env_\pi (tD) \geq \env_\pi (s_j D) = s_j \env_\pi (D)
\]
for any $t_j \geq t \geq s_j$ and $s_j , t_j \in \Q_{\geq0}$.
Therefore we get the result by taking limit.
\end{proof}

\begin{cor}
Let $V \subset \mathrm{Div}_\R (X_\pi)$ be a finite dimensional vector space.
For any divisorial valuation $\nu$ over $X$, the map $D \mapsto \nu (\env_\pi (D))$ is continuous on $V$.
\end{cor}

\begin{proof}
Concave functions on a finite dimensional vector space is continuous, hence the result follows.
\end{proof}

\begin{prop}[\textup{Negativity lemma , \cite[Proposition 2.12]{BdFF}}]
Let $W$ be a $X$-nef $\R$-Weil $b$-divisor over $X$.
Then we have $W \leq \env_\pi (W_\pi)$ for any birational model $\pi$ over $X$.
\end{prop}

\begin{proof}
We fix a birational model $X_\pi$ over $X$.

{\bf Step1.}
When $W$ is a Cartier $b$-divisor determined on $\tau \geq \pi$ and $W_\tau$ is an $X$-globally generated integral Cartier divisor,
then we have $W = Z(\tau_* \sO_{X_\tau} (W_\tau))$.
Since $\tau_* \sO_{X_\tau}(W_\tau) \subset \pi _* \sO_{X_\pi}(W_\pi) $, we have
\begin{eqnarray*}
W = Z(\tau_* \sO_{X_\tau}(W_\tau))  \leq 
      Z(\pi_* \sO_{X_\pi}(W_\pi) )
\leq \env_\pi (W_\pi).
\end{eqnarray*}

{\bf Step2.}
Suppose $W$ is a Carteir $b$-divisor determined on $\tau \geq \pi$ and $W_\tau$ is an $X$-nef $\R$-Cartier $b$-divisor.
Now there exists a sequence of $X$-very ample Cartier divisors ${A_j}$ on $X_\tau$  
, a sequence of positive real numbers ${t_j}$  
and finite dimensional vector space $V \subset \mathrm{Div(W_\tau)} $ such that
$A_j \in V$ and $t_j A_j \to W_\tau $ as coefficient-wise.
By step1, proposition 2.10, and corollary 2.11,
we have
\begin{eqnarray*}
\nu (W) &=& \lim t_j \nu(\bar{A_j}) \\
&\leq& \lim t_j \nu(\env_{\pi} (\bar{A_j})) \\
&=& \lim \nu (\env_\pi (t_j \bar{A_j}))\\
&=& \nu (\env_\pi (W_\pi)) 
\end{eqnarray*}
for any divisorial valuation $\nu$ over $X$, hence $W \leq \env_\pi(W_\pi)$

{\bf  Step3.}
When $W$ is an arbitrary $X$-nef $\R$-Weil $b$-divisor, 
there exists a net of $X$-nef $\R$-Cartier $b$-divisors {$W_i$} such that 
$W_j \to W$ as coefficient-wise and $W_{j,X}$ is contained in a finite dimensional vector subspace of $\mathrm{Div}_\R (X)$.
By step2 and Corollary 2.11, we have 
\begin{eqnarray*}
\nu (W) 
&=& \nu (\lim W_j)\\
&\leq& \nu (\lim \env_\pi (W_{j,\pi}))\\
&=& \nu (\env_\pi (\lim W_{j,\pi}))\\
&=& \nu (\env_\pi (W_\pi))
\end{eqnarray*}
for any divisorial valuation $\nu$ over $X$, hence $W \leq \env_\pi (W_\pi)$. 
\end{proof}

\begin{rem}
We get the usual negativity lemma by Proposition 2.12.
Indeed, let $D$ be an $X$-nef $\R$-Cartier divisor on $X_\pi$ and $\pi_* D \leq 0$.
Then we have
\begin{eqnarray*}
D 
&\leq& \env_X (\pi_* D)_\pi  \leq 0
\end{eqnarray*}
by applying Proposition 2.12 to $W = \bar{D}$.
\end{rem}

\begin{rem}
Let $\pi$ be a birational model of $X$ and $D$ be an $\R$-Weil divisor on $X_\pi$.
Then $\env_\pi(D)$ is the largest element of the set 
\[
\{ W : X{\rm -nef} \  \R{\rm -Weil} \  b{\rm -divisor} \ | \ W_\pi \leq D \}
\]
Indeed, if $W$ is an element of this set, the negativity lemma means
\[
W \leq \env_\pi(W_\pi) \leq \env_\pi(D).
\]
Therefore it is enough to show that the following inequality
\[
\env_\pi (D)_\pi \leq D.
\]
Let $E$ be an prime divisor on $X_\pi$, then we have
\begin{eqnarray*}
\mathrm{ord}_E (Z(\pi_* \sO_{X_\pi}(mD))) &=& -\mathrm{ord}_E ( \pi_* \sO_{X_\pi} (mD)) \\
&=& -\mathrm{min} \{ \mathrm{ord}_E (f) \  | \  f \in \pi_* \sO_{X_\pi}(mD) \} \\
&\leq& \mathrm{ord}_E (mD),
\end{eqnarray*}
hence we have $\env_{X_\pi}(D)_\pi \leq D$.

\end{rem}

\begin{propdef}[\textup{\cite[Definition 2.16]{BdFF}}]
Let $W$ be an $\R$-Weil $b$-divisor over $X$.
Suppose the set $\{$ $Z$ : $X$-nef $\R$-Weil $b$-divisor \ $ |$ \   $Z \leq W \}$ is not empty.
Then it has the largest element.
We call this \emph{nef envelope of W} and denote $\env_\X (W)$.
\end{propdef}

\begin{proof}
By assumption, we get an $X$-nef $b$-divisor $Z \leq W$.
By the negativity lemma and $Z \leq W$, we have 
\[
Z \leq \env_\pi(Z_\pi) \leq \env_\pi (W_\pi)
\]
for any birational model $\pi$.
If $\pi^{'} \geq \pi $, we have $\env_{\pi^{'}}(W_{\pi^{'}}) \geq \env_\pi (W_\pi)$ by definition of nef envelope.
Since the net $\{\env_\pi (W_\pi)\}_{\pi}$ is bounded below and monotonically decreasing,
we can consider the limit $\env_\X (W) = \lim \env_\pi (W_\pi)$ and this satisfies the properties we wanted.
Indeed, only non trivial part is $\env_\X (W) \leq W$, and
it is enough to show that $\env_{X_\pi} (W_\pi)_\pi \leq W_\pi$ for any $\pi$, but it is true by Remark 2.14.
\end{proof}

\section{Intersection numbers of nef $b$-divisors}

In this section, we recall the intersection numbers of nef $b$-divisors defined by \cite{BdFF} and first properties. 
But we also deal the case that $X$ does not have isolated singularities,
therefore Proposition 3.12 is slightly different from \cite[.Proposition 4.16]{BdFF}.

\subsection{The definition of intersection numbers of nef $b$-divisors}

\begin{conv}
In this section, $(R,\m,k)$ is a normal local domain and essentially of finite type over $\C$ 
and $X = \Spec R$.
We denote the unique closed point by $0 \in X$.
\end{conv}

\begin{defn}
Let $W$ be an $\R$-Weil $b$-divisor over $X$.
\begin{enumerate}
\item
$W$ decomposes in the unique way as a sum
\[
W = W^0 + W^{X \backslash 0}
\] 
where all irreducible components of $W^0$ have center $0$, 
and none of $W^{X \backslash 0}$ centered at $0$.

\item
$W$ is a $\R$-Weil $b$-divisor over $0$ if $W = W^0$.

\item
$W$ is bounded below if
there exists $ c \in \R_{>0}$ such that $ W \geq c Z(\m)$

\end{enumerate}
\end{defn}

\begin{rem}
$X$-nef $\R$-Weil $b$-divisors over $0$ are negative by the negativity lemma.
\end{rem}

We say that the following theorem is Izumi's theorem and need the analogous statement Theorem 3.5.

\begin{thm}[\textup{\cite{Izu}}]
Let $\nu$ and $\nu '$ be divisorial valuations on $X$ centered at $0$.
Then there exists $c > 0$ such that 
\[
c^{-1} \nu (f) \geq \nu' (f) \geq c \nu(f)
\]
for any $f \in R$.
\end{thm}

\begin{thm}
Let $\nu$ and $\nu'$ be divisorial valuation centered at $0$.
Then there exists $c >0$ such that
\[
c^{-1} \nu(W) \geq \nu' (W) \geq c \nu(W)
\]
for any $X$-nef $\R$-Weil $b$-divisor over X such that $W \leq 0$.
\end{thm}

\begin{proof}
We have 
\[
W=\env_{\X} (W)=\inf \env_\pi (W_\pi) 
\]
 and
$\env_\pi (W_\pi) \leq 0$ by the negativity lemma.
Therefore we can reduce to $W= \env_\pi (W_\pi)$,
and since $\env_\pi (W_\pi) = \sup \frac{1}{m} Z(\pi_* \sO_{X_\pi} (mW_\pi)) \leq 0$, 
we may assume $W= \frac{1}{m} Z(\pi_* \sO_{X_\pi} (mW_\pi))$ and $\pi_* \sO_{X_\pi} (mW_\pi) \subset \sO_X$.
Thus the theorem follows from the Izumi's theorem.
\end{proof}

\begin{cor}
Let $W$ be an $X$-nef $\R$-Weil $b$-divisor over $X$ such that $W \leq 0$, and
assume that there exists a divisorial valuation $\nu_0$ centered at $0$ such that $\nu_0 (W) \neq 0$.
Then there exists $\epsilon >0$ such that
\[
W \leq \epsilon \cdot Z(\m)
\]
\end{cor}

\begin{proof}
By Theorem 3.5, we have $\nu (W) <0$ for all divisorial valuation $\nu$ centered at $0$.
Let $\pi$ be the normalized blow-up at $\m$.
Then there exists $\epsilon >0 $ such that $W_\pi \leq \epsilon \cdot Z(\m)_\pi$.
By the negativity lemma, we have
\[
W \leq \env_\pi (W_\pi) \leq \epsilon \env_\pi (Z(\m)_\pi) = \epsilon Z(\m)
\]
\end{proof}

\begin{defn}[\textup {cf. \cite[Definition 4.13]{BdFF}}]
Let $W_1 , \dots , W_n $ be $X$-nef $\R$-Weil $b$-divisors over $0$.
We set
\[
W_1 \cdot \ldots \cdot W_n = \inf_{C_i \geq W_i} (C_1 \cdot \ldots \cdot C_n)
\]
where the infimum is taken over all $X$-nef $\R$-Cartier $b$-divisors over $0$
such that $C_i \geq W_i$.
\end{defn}

\begin{prop}
Let $W_1 , \dots , W_n $ and $W'_1, \dots , W'_n$ be $X$-nef $\R$-Weil $b$-divisors over $0$.

\begin{enumerate}
\item
If $W_i \leq W'_i$ for all $i=1 , \dots , n$, then we have 
$W_1 \cdot \ldots \cdot W_n \leq W'_1 \cdot \ldots \cdot W'_n$.

\item
If $W_1 , \dots , W_n$ are bounded below, then we have 
$W_1 \cdot \ldots \cdot W_n > -\infty$.

\item
$W_1 \cdot \ldots \cdot W_n = 0 \Rightarrow W_1 = \cdots = W_n = 0$

\end{enumerate}
\end{prop}

\begin{proof}
We have
\[
W_1 \cdot \ldots \cdot W_n \leq C_1 \cdot \ldots \cdot C_n
\]
for all $X$-nef $\R$-Cartier $b$-divisors over $0$ such that $C_i \geq W'_i$
by  the definition of intersection numbers of $b$-divisor.
Thus we have the statement (1).

(2) is clear by the definition of bounded below and (1).

If $W_i \neq 0 $ for all $i$, 
we can find $\epsilon > 0$ such that $W_i \leq \epsilon \cdot Z(\m)$
by using corollary 3.6.
Therefore we have 
\[
W_1 \cdot \ldots \cdot W_n \leq \epsilon^n Z(\m)^n < 0.
\]
\end{proof}

\subsection{Properties of intersection numbers under finite morphisms}
In this subsection, we discuss the properties of intersection numbers under finite morphisms
to consider endomorphisms.

\begin{conv}
In this subsection, $X$ is as in convention 3.1.
and $(S,\n,l)$ is also a normal local domain and essentially of finite type over $l$ 
and $Y = \Spec S$.
We denote also the unique closed point by  $0 \in Y$.
Let $f  : (Y,0) \to (X,0) $ be finite surjective morphism satisfying $f^{-1} (0) = {0}$ 
and $\tilde{d} = \displaystyle{\frac{\mathrm{deg}(f)}{[l : k]} }$ where $[l : k]$ is the index of the field extension of the residue fields.
Now, $f$ induces linearly maps $f_* : \mathrm{Div}(\Y) \to \mathrm{Div}(\X)$ and
$f^* : \mathrm{Div}(\X) \to \mathrm{Div}(\Y)$.
\end{conv}

\begin{prop}
\ 
\begin{enumerate}

\item
Let $W$ be an integral Cartier $b$-divisor over $X$ determined by a birational model $\pi$.
Then $W_\pi$ is $X$-globally generated if and only if 
there exists a fractional ideal sheaf $\fa$ of $X$ such that $Z(\fa) = W$.

In particular the Cartier divisors $Z(\fa)$ with $\fa$ ranging over all ideal sheaves of $X$ 
generates $\mathrm{CDiv(\X)}$ as a group.

\item
If $W$ is an $\R$-Cartier $b$-divisor over $Y$, then $f_* W $ is also a $\R$-Cartier $b$-divisor.

\item
If $W$ is a $Y$-nef $\R$-Weil $b$-divisor , $f_* W $ is an $X$-nef $\R$-Weil $b$-divisor.
\end{enumerate}
\end{prop}

\begin{proof}
Let $W$ be a Cartier $b$-divisor over $X$ determined on $\pi$ and $W_\pi$ is $X$-globally generated.
Since the natural map
$\pi^* \pi_* \sO_{X_\pi} (W_\pi) \to \sO_{W_\pi} (W_\pi)$
is surjective,
we see that $W=Z(\fa)$ with $\fa := \pi_* \sO_{X_\pi} (W_\pi)$.
The second assertion of (1) follows from the fact that any Cartier divisor on a given model $X_\pi$ 
can be written as a difference of two $\pi$-very ample Cartier divisors.

Next we will prove (2) so let $W$ be an $\R$-Cartier $b$-divisor over $Y$.
By the statement of (1), we may assume $W = Z(\fa)$ where $\fa$ is an ideal sheaf of $Y$.
In fact that we have
\[
f_* Z(\fa) = Z(N_{Y/X} (\fa) )
\]
where $N_{Y/X} (\fa) $ denotes the image of $\fa$ under the norm homomorphism.
We omit the proof of this fact. See {\cite[Proposition 1.14]{BdFF}} for the details.

Next we will prove (3) so let $W$ be a $Y$-nef $\R$-Weil $b$-divisor over $Y$.
By the fact of $W=\env_\X (W)$ and the definition of nef envelope,
$W$ is the limit of a net of $X$-globally generated $b$-divisors $\{r_\pi Z(\fa)\}$ where 
$r_\pi$ is a positive real number.
By the proof of the assertion of (2),  we have $f_* W = \lim {r_\pi Z(N_{Y/X} (\fa))}$ and this is 
an $X$-nef $\R$-Weil $b$-divisor.

\end{proof}

\begin{prop}
Let $W$ be an $\R$-Weil $b$-divisor over $X$ and $W$ has the nef envelope.
Then $f^* W$ has the nef envelope and we have
\[
\env_\Y (f^*W) =f^* \env_\X (W).
\]
We similarly have 
\[
\env_Y (f^* D) = f^* \env_X (D)
\]
for any $\R$-Weil divisor $D$ on $X$.
\end{prop}

\begin{proof}
Since $f^* \env_\X (W) \leq f^*W$ and $f^* \env_\X (W)$ is $Y$-nef,
$f^* W$ has the nef envelope and we have
\[
f^* \env_\X (W) \leq \env_\Y (f^*W) \leq f^*W.
\]
Since $f_* \env_\Y (f^*W)$ is $X$-nef by Proposition 3.10 (3),
we have 
\[
f_* \env_\Y (f^*W) \leq \env_\X (f_* f^* W) = \mathrm{deg}(f) \env_\X (W).
\]
Thus we have $f_* \env_\Y (f^*W) = \env_\X (W)$ 
and since  $f^* \env_\X (W) \leq \env_\Y (f^*W)$,
this inequality is equal.

\end{proof}

\begin{prop}
Let $W_1 , \ldots W_n$ be $X$-nef $\R$-Weil $b$-divisors over 0.
Then we have
\[
f^* W_1 \cdot \ldots \cdot f^* W_n \leq \tilde{d} \cdot (W_1 \cdot \ldots \cdot W_n)
\]
\end{prop}

\begin{proof}
We take any $X$-nef $\R$-Cartier $b$-divisor $C_i$ over 0 such that $W_i \leq C_i$ for all $i=1, \ldots , n$.
Since $f^* W_i \leq f^* C_i $ and $f^* C_i$ is $Y$-nef $\R$-Cartier $b$-divisors over 0 and the projection formula,
we have
\begin{eqnarray*}
f^* W_1 \cdot \ldots \cdot f^* W_n &\leq& f^* C_1 \cdot \ldots \cdot f^*C_n \\
&=& \tilde{d} \cdot (C_1 \cdot \ldots \cdot C_n )
\end{eqnarray*}
where $f^* C_1 \cdot \ldots \cdot f^*C_n$ is the intersection number considering $Y$ as a scheme over $l$ 
and $C_1 \cdot \ldots \cdot C_n $ is the intersection number considering $X$ as a scheme over $k$.
Taking infimum, we have
\[
f^*W_1 \cdot \ldots \cdot f^* W_n \leq \tilde{d} \cdot (W_1 \cdot \ldots \cdot W_n)
\]
\end{proof}

\begin{rem}
If $X$ and $Y$ have an isolated singularities, then the inequality in Proposition 3.11 is equal by \cite{BdFF}[Proposition 4.16.].
But we do not know that this inequality is equal or not in general cases.
\end{rem}

\section{The isolated volume and log discrepancy of non $\Q$-Gorenstein varieties}

In this section, we introduce log discrepancies of non $\Q$-Gorenstein varieties by using the nef envelope
in place of the pullback. 
Here, we can consider two nef envelopes $\env_X(K_X)$ and $-\env_X(-K_X)$,
but we use $\env_x (K_X)$ and we state it define a good singularities with many meaning.
Next, we define the isolated volume characterizing valuative log canonical singularities and 
we observe the properties of this invariant under  finite morphisms.
Finally we will prove the local statement Theorem 1.3 corresponding to the main theorem.

\subsection{Log discrepancy with respect to $\env_X (K_X)$ and related singularities}
In this subsection, $X$ is as in Convention 2.1.

\begin{defn}[\textup{cf. \cite[Definition 3.4]{BdFF} and \cite[Definition 3.2]{dFH}}]
The \emph{log-discrepancy $b$-divisor} is defined by
\[
\AX^+ := K_{\X} + 1_{\X/X} - \env_X(K_X)
\]
where  $(1_{\X/X})_\pi$ is the sum of exceptional prime divisors on $X_\pi$, 
and $1_{\X/X}$ is the projective limit of $\{(1_{\X/X})_\pi \}$.
\end{defn}

\begin{rem}
In \cite{BdFF}, log-discrepancy $b$-divisor is $ K_{\X} + 1_{\X/X} + \env_X(-K_X) $ and 
it is not equal to $\AX^+$ in general.
But by Section $6$, if $X$ is numerically $\Q$-Gorenstein, then these are equal,
in particular if $X$ is a surface, these give same definitions.
\end{rem}

\begin{defn}
$X$ has \emph{valuative log canonical singularities} if $\AX^+ \geq 0$.
\end{defn}

Next we discuss this singularities.
If $X$ is $\Q$-Gorenstein, then this is the log canonical singularity,
but if $X$ is not $\Q$-Gorenstein, this singularity is related to the singularities in 
\cite{Ful} or \cite{BH}.

\begin{thm}
The following are equivalence to each other.
\begin{enumerate}
\item
$X$ has valuative log canonical singularities.

\item
For any birational model $X_\pi$ over $X$,
we have
\[
\pi_* \sO_{X_\pi} (m(K_{X_\pi} + E^\pi)) = \sO_X(mK_X)
\]
for any positive number $m$, 
where $E^\pi$ is the sum of the exceptional prime divisors on $X_\pi$.

\ \newline
If $X$ has the log canonical model, next condition is also equivalent.
\item
The log canonical model of $X$ is small i.e, isomorphism in codimension one.
\end{enumerate}
\end{thm}

\begin{proof}
\ 

Proof of $\underline{(1) \Rightarrow (2)} $

$\pi_* \sO_{X_\pi} (m(K_{X_\pi} + E^\pi)) \subset \sO_X(mK_X)$ is clear , so we prove the converse inclusion.
We take $f \in \sO_X(mK_X)$, then 
\[
\mathrm{div} (f) + Z(\sO_X (mK_X)) \geq 0
\]
Since $\env_X(K_X) $ is the suprimum of $\{ \displaystyle{\frac{1}{m} Z(\sO_X(mK_X))} \}$,
we have
\[
\frac{1}{m} \mathrm{div}(f) + \env_X (K_X) \geq 0
\]
By the condition (1),
we have $K_{X_\pi} + E^\pi \geq \env_X(K_X)_{X_\pi}$, so we have
\[
\frac{1}{m} \mathrm{div}_{X_\pi}(f) + K_{X_\pi} + E^\pi \geq 0
\]
It means $f \in \pi_* \sO_{X_\pi}(m(K_{X_\pi} + E^\pi)) $.

Proof of $\underline{(2) \Rightarrow (1)}$

By the condition (2), $\env_X(K_X) = \env_{X_\pi}(K_{X_\pi} + E^\pi)$.
So it is enough to show $\env_{X_\pi}(D)_\pi \leq D$ for any Weil divisor $D$ on $X_\pi$.
It follows from Remark 2.14.

Next we assume $X$ has the log canonical model $X_\pi$.
Now the log canonical model satisfies following conditions.
\begin{itemize}
\item $K_{X_\pi} + E^\pi$ is a $X$-ample divisor.
\item $(K_{X_\pi} , E^\pi)$ is a log canonical pair.
\end{itemize}
where $E^\pi$ is the sum of the exceptional prime divisors.

Proof of $\underline{ (3) \Rightarrow (1)}$

Since $\pi$ is small, we have
\[
\pi_* \sO_{X_\pi} (mK_\pi) =\sO_{X} (mK_X)
\]
for any $m$ and there exists $m$ such that
\[
\bar{K_{X_\pi}} = \frac{1}{m}Z(\pi_* \sO_{x_\pi}(mK_{X_\pi}))
\]
by the ampleness of $K_{X_\pi}$.
Since we have $\env_X (K_X) = \bar{K_{X_\pi}}$ and $X_\pi$ is log canonical,
we have $\AX^+ \geq 0$.

Proof of $\underline {(2) \Rightarrow (3)}$

By the ampleness of $K_{X_\pi} + E^\pi$ and the condition (2),
we have
\begin{eqnarray*}
X_\pi &=& \Proj \bigoplus \pi_* \sO_{X_\pi}(m(K_{X_\pi} + E^\pi) ) \\
&=& \Proj \bigoplus \sO_X (mK_X),
\end{eqnarray*}
and it is small by {\cite[Lemma 6.2]{KM}}.

\end{proof}

\begin{rem}
In {\cite[Example 5.4]{BdFF}}, they gave the variety 
satisfying (2) in Theorem 4.4 but not satisfying
\[
\AX := K_{\X} + 1_{\X/X} + \env_X(-K_X) \geq 0
\].
In {\cite[Theorem 4.3]{Zha14}} and {\cite[Example 4.10]{Has}},
they gave the variety
satisfying
\[
\AX = K_{\X} + 1_{\X/X} + \env_X(-K_X) \geq 0
\]
but not satisfying
\[
A_{\X/X. m} := K_{\X} + 1_{\X/X} + \frac{1}{m} Z(\sO_X (-mK_X)) \geq 0
\] 
for any positive integer $m$.

\end{rem}

\begin{prop}
Let $X_\pi$ be a log resolution of $X$.
Then $\AX^+ \geq 0$ if and only if $(\AX^+)_\pi =K_{X_\pi} + E^\pi - \env_X(K_X)_\pi \geq 0$.
In particular
the non valuative log canonical set
\[
\mathrm{Nvlc} :=\bigcup_{\ord_E (\AX^+) < 0} \mathrm{c}_X (E)
\]
is a Zariski closed subset.
\end{prop}

\begin{proof}
Let $X_\pi$ be a log resolution of $X$.
Since the union of the element of
\[
 \{ \mathrm{c}_X(E) \ | \ \mathrm{ord}_E(K_{X_\pi} + E^\pi - \env_X(K_X)_\pi) < 0
 \mbox{, E : prime divisor on} \  X_\pi \}
\]
is a closed subset of $X$,
considering the complement of this set,
the second assertion follows from first assertion.
Thus we may assume $(\AX^+)_\pi \geq 0$.
Let $X_\mu$ be the birational model such that $\mu \geq \pi$ and we set $\nu : X_\mu \to X_\pi$.
Since $\env_X(K_X) $ is also a $X_\pi$-nef $b$-divisor over $X_\pi$,
we can use the negativity lemma over $X_\pi$ and we have
\[
\env_X(K_X) \leq \env_{X_\pi} (\env_X(K_X)_\pi) = \bar{\env_X(K_X)_\pi}
\]
Since the pair $(X_\pi , E^\pi) $ is log canonical,
we have
\begin{eqnarray*}
(A_{\X/X}^+)_\mu
&=& K_{X_\mu} + E^\mu - \env_X(K_X)_\mu \\
&\geq& K_{X_\mu} + E^\mu - \nu^* (\env_X(K_X)_\pi) \\
&=& K_{X_\mu} + E^\mu - \nu^*(K_{X_\pi} + E^\pi) + \nu^*(K_{X_\pi} + E^\pi - \env_X(K_X)_\pi) \\
&\geq& 0
\end{eqnarray*}
\end{proof}

\subsection{The isolated volume and local version of main theorem}

In this subsection, $X$ is as in Convention 3.1.

\begin{defn}[\textup{cf. \cite[Definition 4.18]{BdFF}}]
If $\AX^+$ is bounded below which is defined in Definition 3.2, 
the isolated volume of $X$ defined by
\[
\Vol^+ := - ( \env_\X (\AX^+) )^n
\]
\end{defn}

\begin{rem}
It is also not equal to the isolated volume in \cite{BdFF} in general.
And if $\AX^+$ is bounded below, we can find $c >0$ such that $cZ(\m) \leq \AX^+$.
Then there exists $\env_\X (\AX^+)$ and $0 \leq \Vol^+ \leq c^n \cdot (- Z(\m)^n) < \infty$.
So  $\Vol^+$ is a positive finite number.
\end{rem}

By the following proposition, $\Vol^+$  characterizes the valuative log canonical singularities.

\begin{prop}
If $\AX^+$ is bounded below, 
then $\AX^+ \geq 0$ if and only if $\Vol^+ = 0$.
\end{prop}

\begin{proof}
By Proposition 3.8 (3), the right-hand side is equivalent to $\env_\X(\AX^+) =0$.
Since $\AX^+ \geq \env_\X (\AX^+)$, $\Vol^+ = 0$ implies $\AX^+ \geq 0$.
Next we assume $\AX^+ \geq 0$,
since $0$ is a $X$-nef $b$-divisor, $\env_\X(\AX^+) \geq 0$.
But by the negativity lemma, $\env_\X(\AX^+) \leq 0$.
\end{proof}

\begin{prop}
If $(\AX^+)^{X \backslash 0} \geq 0$ defined in Definition 3.2, then $\AX^+ $ is bounded below.
\end{prop}

\begin{proof}
Let $\pi$ be a log resolution of $X$.
Since the negative part $C$ of $K_{X_\pi} + E^\pi - \env_X(K_X)_\pi $ is centered at $0$
and $\bar{C}$ is bounded below, it is enough to show $\AX^+ \geq \bar{C}$.
By similar argument in proof of Proposition 4.5, we have
\begin{eqnarray*}
\AX^+ &=& K_\X + 1_{\X/X} - \env_X(K_X) \\
&\geq& K_\X + 1_{\X/X} - \bar{K_{X_\pi} + E~\pi} + \bar{K_{X_\pi} + E^\pi - \env_X (K_X)_\pi} \\
&\geq& \bar{C}
\end{eqnarray*}
\end{proof}

\begin{prop}
Let $(X , 0)$ , $(Y , 0)$ , $f : (Y , 0) \to (X , 0)$ and $\tilde{d}$ be as in Convention 3.9,
then we have $\AY^+ \leq f^* \AX^+$.
In particular if $\AX^+$ and $\AY^+$ are bounded below,
then we have $\tilde{d} \cdot \Vol^+ \leq \mathrm{Vol}^+ (Y , 0) $.
\end{prop}

\begin{proof}
Since the ramification divisor is the difference of the canonical divisor of the domain and the pullback of canonical divisor of the target,
we have 
\[
K_\Y + 1_{\Y/Y} - f^* (K_\X + 1_{\X/X}) = R_f
\]
where $R_f$ is the $b$-divisor defined by the strict transforms of ramification divisor of $f$.
So we have
\begin{eqnarray*}
f^*\AX^+ - \AY^+ &=& -R_f - f^* \env_X (K_X) + \env_Y (K_Y) \\
&\geq& -R_f + \env_Y (R_f) \\
&\geq& 0
\end{eqnarray*}
where second inequality followed by
\begin{eqnarray*}
\env_Y (K_Y) &=& \env_Y (K_Y - f^*K_X + f^*K_X ) \\
&\geq& \env_Y (R_f) + \env_Y (f^*K_X) \\
&=& \env_Y (R_f) + f^* \env_X (K_X)
\end{eqnarray*}
and third inequality follows from 
the maximality of nef envelope of $R_f \geq0$, $0$ is a $X$-nef $b$-divisor and 
$\env_Y((R_{f,Y})_Y = R_{f,Y}$

Using Proposition 3.12 and Proposition 3.8 , we have
\[
(\AY^+)^n \leq (f^* \AX^+)^n \leq \tilde{d} \cdot (\AX^+)^n
\]
Thus we see the result.
\end{proof}

\begin{thm}[\textup{Theorem 1.3}]
Let $(R , \m , k)$ be a normal local ring essentially of finite type over $\C$ and 
we set $X := \Spec R$ and we denote the unique closed point $0$.
Assume $X$ has an endomorphism $f : X \to X$ with
$f^{-1}(0) = {0} $ and $\tilde{d} = \displaystyle{ \frac{\mathrm{deg}(f)}{[f_* k : k]}} > 1$.
Furthermore, we assume $X \backslash 0$ has valuative log canonical singularities.\newline
Then $X$ has valuative log canonical singularities.

\end{thm}

\begin{proof}
First, $\AX^+$ is bounded below by Prop 4.9 and the assumption $(\AX^+)^{X \backslash 0} \geq 0$.
Thus we can apply Proposition 4.10 and we have 
\[
\tilde{d} \cdot \Vol^+ \leq \Vol^+.
\]
Since $\Vol^+$ is a finite positive number and $\tilde{d} > 1$,
we have $\Vol^+ = 0$.
This theorem followed from the above result and Proposition 4.8.
\end{proof}

\section{Proof of Main theorem}
In this section, we prove the following theorem.
Combining Theorem 5.1 and Theorem 4.10, we have an affirmative answer to \cite[Conjecture 1.5]{BH}.
Moreover, we see the Theorem 1.4 as the direct corollary of Theorem 5.1.

\begin{thm}[\textup {cf. \cite[Theorem 1.2]{BH}}]
Let $X$ be a normal complex variety over $X$ admitting an endomorphism $f : X \to X$.
Let $Z$ be an irreducible component of non valuative log canonical locus.\newline
Then up to replacing $f$ by some iterate, $Z$ is totally invariant, 
and the induced endomorphism satisfies 
\[
\deg (f |_Z) = \deg (f).
\]

\end{thm}

\begin{proof}
We set the non valuative log canonical locus 
\[
\mathrm{Nvlc} :=\bigcup_{\ord_E (\AX^+) < 0} \mathrm{c}_X (E) ,
\]
and first, $\mathrm{Nvlc}$ is a closed subset of $X$ by Proposition 4.5,
thus it has finite irreducible components.
We prove the first statement by the induction on codimension of $Z$.
It is clear if $\mathrm{codim}Z = 0$ since $\mathrm{Nvlc} \neq X$.

\begin{cl}
$f^{-1}(Z)$ is an irreducible component of $\mathrm{Nvlc}$ up to replacing $f$ some iterate.
\end{cl}

\begin{clproof}
We take the generic point $\eta$  of $Z$ and any prime divisor $E$ over $X$ centered at $\eta$ with 
$\mathrm{ord}_E (\AX^+) < 0$.
We denote 
\[
f^* E = r_1 F_1 + \cdots + r_m F_m
\]
where $r_i$ is the ramification index of $F_i$,
then we have
\[
\ord_{F_i}(\AX^+) \leq \ord_{F_i}(f^* \AX^+) = r_i \ord_E (\AX^+) <0
\]
So  $f^{-1}(Z) = \bigcup \bar{\mathrm{c}_X(F_i)}$ is contaied in $\mathrm{Nvlc}$.
Suppose $f^{-1}(Z) $ is not irreducible component of $\mathrm{Nvlc}$, 
then we can take $Z'$ the irreducible component of $\mathrm{Nvlc}$ containing $f^{-1}(Z)$.
Using induction hypothesis, we can assume $Z'$ is totally invariant of $f$.
Since $f(Z') = Z'$ and $f^{-1}(Z) \subset Z'$ , $Z'$ is an irreducible component containing $Z$.
But $Z$ is also an irreducible component, so it is a contradiction.
\end{clproof}

By claim, we get a one-to-one correspondence of irreducible components having same dimension as 
$Z$ by $f^{-1}$.
Hence replacing $f$ some iterate, $Z$ is a totally invariant set of $f$.

Next, we prove the second statement by reducing to the local case.
Let $\eta$ be the generic point of $Z$ and $f_\eta : (X_\eta , \eta) \to (X_\eta , \eta ) $ be the localization of $f$.

Since $X_\eta \backslash \{ \eta \}$ has valuative log canonical singularities 
and $X_\eta$ is not valuative log canonical,
it follows that $\tilde{d}$ is equal to $1$ from Theorem 4.10.
Here, we have
\[
\deg (f_\eta) = \deg (f) = [{f_\eta}_* \kappa (Z) : \kappa (Z)]
\]
but $[{f_\eta}_* \kappa (Z) : \kappa (Z)]$ is nothing but $\deg (f |_Z)$.
Thus we conclude
\[
\deg (f |_Z) = \deg (f).
\]
\end{proof}

As the direct corollaries, we get the following two results.

\begin{thm}
Let $X$ be a normal complex projective variety over $k$ admitting a non-invertible polarized endomorphism.
Then we have $X$ has valuative log canonical singularities.
In particular if $X$ has the log canonical model, then it is small.
\end{thm}

\begin{cor}
Let $X$ be as in Theorem 5.2.
Then there exists a normal affine cone  $\tilde{X}$ of $X$ such that 
$\tilde{X}$ has the valuative log canonical singularities.
In particular if $\tilde{X}$ has the log canonical model, then it is small.
\end{cor}

\begin{proof}

From now, $X$ is as in Theorem 5.2 and $f : X \to X$ is a non-invertible polarized endomorphism.
We can take an ample divisor $H$ on $X$ such that 
$ f^*H = qH$ for some positive integer $q$ and the affine cone $\tilde{X}$ associated to $H$ is normal and we denote vertex of $\tilde{X}$ by $0$.

Suppose 
$\mathrm{Nvlc} \neq \emptyset$.
We take an irreducible component $Z$, 
then we may assume $Z$ is totally invariant of $f$ and $\deg (f |_Z) = \deg (f)$
by Theorem 5.1.
Since $f^*H = qH$, We get next two equations
\[
\deg(f) \cdot H^{\dim X} = (f^* H )^{\dim X} = q^{\dim X} \cdot H^{\dim X}
\]
\[
\deg(f|_Z) \cdot (H|_Z)^{\dim Z} = ({f|_Z}^* {H|_Z})^{\dim Z} = q^{\dim Z} \cdot (H|_Z)^{\dim Z}
\]
However it contradicts to $\deg (f |_Z) = \deg (f)$.

Finally, we prove Corollary 5.3.
Theorem 5.2 implies that $\tilde{X} \backslash 0$ has valuative log canonical singularities.
If $\tilde{X}$ is not valuative log canonical, then $\deg (f) = \deg(f |_{ \{0\} })$,
but it is contradiction.

\end{proof}

\begin{rem}
In \cite{SS}, if $X$ is of log Calabi-Yau type, then every affine cone have log canonical singularities
in the sense of \cite{dFH}, in particular it has valuative log canonical singularities.
Therefore, Corollary 5.3 gives a weakly type of an affirmative answer to Conjecture 1.5.
\end{rem}

\section{Numerically $\Q$-Cartier divisors and other applications of intersection numbers}

In this section, we consider valuative properties of numerically $\Q$-Cartier divisors.
Furthermore, we show that normal projective varieties admitting \'etal in codimension one polarized endomorphism 
are $\Q$-Gorenstein and the affine cone is also $\Q$-Gorenstein.
If $X$ is $\Q$-Gorenstein, it is well known.

\subsection{Numerically $\Q$-Cartier divisors}

In this subsection, We survey basic properties of numerically $\Q$-Cartier divisors. 
See \cite{BdFFU} for the details.
The purpose of this subsection is to prove the following proposition.

\begin{prop}[\textup{\cite[Proposition 5.9]{BdFFU}}]
Let $D$ be a $\Q$-Weil divisor on $X$.
The followings are equivalent
\begin{enumerate}
\item
$D$ is numerically $\Q$-Cartier divisor.

\item
$\env_X(D) + \env_X(-D) =0$
\end{enumerate}
\end{prop}

We will give a slightly different proof of \cite{BdFFU}.
First we recall the definition of numerically $\Q$-Cartier divisors.

\begin{defn}[\textup{\cite[Definition 2.26]{BdFF}}]
Let $X$ be a normal integral scheme essentially of finite type over a field and $D$ be a $\Q$-Weil divisor.
$D$ is a numerically $\Q$-Cartier divisor if
there exists a birational model $\pi$ and $\Q$-Weil divisor $D'$ on $X_\pi$ such that
$D'$ is a $X$-numerically trivial $\Q$-Cartier divisor and $\pi_* D' = D$.

If $K_X$ is a numerically $\Q$-Cartier divisor, we say $X$ is numerically $\Q$-Gorenstein.

\end{defn}

\begin{lem}
Let $X$ be a normal variety,  $X_\pi$ be a birational model of $X$ and 
$D$ be a $X$-movable divisor on $X_\pi$.
Then we have 
\[
\env_\pi(D)_\pi = D
\]
\end{lem}

\begin{proof}
By Corollary 2.11, 
\[
\lim \env_\pi(D + \frac{1}{m} H) = \env_\pi(D)
\]
for any $X$-ample divisor $H$.
So we may assume $D$ is $X$-mobile and for some $m$, the natural map
\[
\pi^* \pi_* \sO_{X_\pi}(mD) \to \sO_{X_\pi}(mD)
\]
is surjective at codimension one point.
It means
\[
Z(\pi_* \sO_{X_\pi}(mD))_\pi = mD
\]
and this lemma follows from this equation.

\end{proof}

\begin{cor}
Let $X$ be a normal variety and $W$ be a $\R$-Weil $b$-divisor.
Then $W$ is $X$-nef $b$-divisor if and only if
$W_\pi$ is a $X$-movable divisor for any smooth model $X_\pi$.
\end{cor}

\begin{proof}
If $W$ is $X$-nef, we may assume $W$ is determined by an birational model $\pi$ 
and $W_\pi$ is $X$-nef.
If the push-forward of a $X$-nef divisor is $\R$-Cartier , then it is $X$-movable.

Next we consider that the condition of $W_\pi$ is a $X$-movable divisor for any smooth model $X_\pi$.
By the previous lemma, 
\[
\env_\pi(W_\pi)_\pi = W_\pi
\]
for any smooth model $\pi$.
In particular
\[
\env_{\X}(W) = \inf \env_\pi(W_\pi) = W
\]
so $W$ is $X$-nef.
\end{proof}

\begin{proof}[Proof of Proposition 6.1]\ 

First we will prove (1) $\Rightarrow$ (2).

By the largest property of nef envelope, if $D'$ is $X$-numerically trivial,
$\bar{D'} \leq \env_X(D)$ and $-\bar{D'} \leq \env_X(-D)$ and we have $\env_X(D) + \env_X(-D) \geq 0$.
In general since $\sO_X(mD) \sO_X(-mD) \subset \sO_X$ and
\[
\env_X(D) + \env_X(-D) = \lim {\frac{1}{m} (Z(\sO_X(mD)) + Z(\sO_X(-mD))}
\]
we have $\env_X(D) + \env_X(-D) \leq 0$.

Next we will prove the converse implication.
Let $X_\pi$ be a resolution of singularities of $X$.
We take a smooth model $X_{\pi'}$ such that $\pi' \geq \pi$ and we set $\mu : X_{\pi'} \to X_\pi$
Since $-\env_X(D)_{\pi'}$ is $X$-movable,
\[
\mu^*\env_X(D)_\pi - \env_X(D)_{\pi'}
\]
is a $X_\pi$-movable divisor.
In particular 
\[
\bar{\env_X(D)_\pi} - \env_X(D)
\]
is a $X_\pi$-nef $b$-divisor.
Therefore we can use negativity lemma and we have
\[
\bar{\env_X(D)_\pi} - \env_X(D) \leq 0
\]
Furthermore, by the negativity lemma,
\[
\env_X(D) \leq \bar{\env_X(D)_\pi}
\]
Thus $\env_X(D)$ is a $\Q$-Cartier $b$-divisor and we set $D' = \env_\pi(D)_\pi$.
Since $\mu^*D'$ is $X$-movable for any $\mu$ and $\pi'$, 
$D'$ is $X$-nef.
By similar argument, we get $-D'$ is also $X$-nef.
These imply that $D'$ is numerically $\Q$-Cartier.

\end{proof}

\subsection{\'Etale in codimension one endomorphism}
\ 

First, we consider the local situation.

\begin{thm}
Let $X$ be a normal integral scheme essentially of finite type over $\C$ and we fix a closed point $0 \in X$
and $f : (X , 0) \to (X , 0)$ is \'etale in codimension $1$ endomorphism with $f^{-1}(0) = {0} $.
Assume $X \backslash 0$ is $\Q$-Gorenstein and $\tilde{d}(f) > 1$.
Then $X$ is numerically $\Q$-Gorenstein.
\end{thm}

\begin{proof}
It is enough to show 
\[
\env_X(K_X) + \env_X(-K_X) = 0
\]
by Proposition 6.1.
Since 
\[
\frac{1}{m}Z(\sO(mK_X) \sO_X(-mK_X)) \leq \env_X(K_X) + \env_X(-K_X)
\]
and  $X \backslash 0$ is $\Q$-Gorenstein,
$\env_X(K_X) + \env_X(-K_X) $ is bounded below and a $X$-nef divisor over $0$.
So we can consider the intersection number
\[
\Gamma : = -(\env_X(K_X) + \env_X(-K_X))^n
\]
Since $f$ is \'etale in codimension $1$, $f^*K_X = K_X$ and we have
\[
f^*(\env_X(K_X) + \env_X(-K_X)) = \env_X(K_X) + \env_X(-K_X)
\]
By Proposition 3.12, we also have
\[
\Gamma \geq \tilde{d}(f) \cdot \Gamma
\]
So the result follows.
\end{proof}

\begin{prop}[\textup{cf. \cite[Lemma 2.2]{BdFF} , \cite[Theorem 4]{Cut}}]
Let $X$ be a normal integral scheme essentially of finite type over $\C$ and
$D$ be numerically $\Q$-Cartier.
If the following sheaf
\[
\bigoplus_m \sO_X (mD)
\]
is finitely generated $\sO_X$-algebra.
Then $D$ is $\Q$-Cartier.
\end{prop}

\begin{proof}
We consider the small projective birational morphism
\[
\pi : \Proj_X (\bigoplus_m \sO_X (mD)) = Y \to X
\]
by  {\cite[Lemma 6.2]{KM}}.
Then $\sO_X (mD) \cdot  \sO_Y$ is invertible sheaf for divisible enough $m$.
We set the divisor $\Gamma_m$ on $Y$ as the following
\[
\sO_X (mD)  \sO_Y = \sO_Y (m\Gamma_m)
\]
By assumption, we have $\Gamma_m$ = $\Gamma_{m'}$ for divisible enough $m$ and $m'$.
This means that
\[
\env_X(D) = \lim \frac{1}{k} Z(\sO_X(kD)) = \lim \bar{\Gamma_k} = \bar{\Gamma_m}
\]
But $\Gamma_m$ is $X$-ample and $\env_X (D)_Y$ is $X$-numerically trivial.
Indeed if $\env_X (D)$ is $X$-numerically trivial and $\env_X(D)_Y$ is $\Q$-Cartier,
then $\env_X(D) = \bar{\env_X (D)_Y}$.
So $\pi$ is identity map and $mD$ is Cartier for some $m$.

\end{proof}

We use the following fact by proving that
the non numerically $\Q$-Gorenstein locus is closed.

\begin{thm}(cf. {\cite[Theorem 1.1 and Theorem 4.9]{Has}})
Let $X$ be a normal variety over $\C$.
Then there exists a birational model $h : W \to X$ and $0<t<1$ such that
\begin{itemize}
\item any prime divisor $E$ on $W$ satisfies 
\[
\ord_E (\AX) = \ord_E ( K_{\X} + 1_{\X/X} + \env_X(-K_X) ) < 0,
\]
\item $ (W,E^h)$ is log canonical where $E^h$ is the sum of the exceptional prime divisors, and
\item $K_W + (1-t) E^h $ is $h$-ample.
\end{itemize}
\end{thm}

In {\cite{Has}}, he defined the log discrepancies for pairs not necessary $\R$-Cartier,
and it coincides with $\ord_E (\AX)$ for any prime divisor $E$ over $X$ by {\cite[Theorem 4.9]{Has}}.
Furthermore, in the proof of {\cite[Theorem 1.1]{Has}}, he constructed a model as in Theorem 6.8.

\begin{defn}
Let $X$ be a normal integral scheme essentially of finite type over $\C$.
We set the non numerically $\Q$-Gorenstein locus
\[
\mathrm{Nnum} (X) := \bigcup \mathrm{c}_X (E)
\]
where the union is taken by the set 
\[
\{ E : exceptional \  prime \  divisor \  over \  X \ | \ \ord_E ( \env_X(K_X) + \env_X(-K_X) ) < 0\}. 
\]
\end{defn}

\begin{thm}
Let $X$ be a normal valuative log canonical variety over $\C$.
Then the non-numerically $\Q$-Gorenstein locus coincides with the non-$\Q$-Gorenstein locus.
In particular, non-numerically $\Q$-Gorenstein locus is closed.
\end{thm}

\begin{proof}
We take a birational model $h : W \to X$ as in Theorem 6.7 for some $t$.
Let $Z$ be the non $\Q$-Gorenstein locus and
$U$ be the complement of the image of the exceptional prime divisors on $W$.
By Theorem 6.7, $h$ is the log canonical model over $U$ after restricting on $U$.
Since $h$ has the small log canonical model on $U$, we have
\[
\bigoplus_m \sO_U (mK_U)
\]
is finitely generated.
Therefore, by Proposition 6.6, 
$\mathrm{Nnum} (U)$ coincides with non $\Q$-Gorenstein locus.
Because of $N_{num} (X) \cap U = \mathrm{Nnum} (U) $,
it is enough to show that $X \backslash U \subset \mathrm{Nnum}$

Let $E$ be an exceptional prime divisor on $W$, we have
\[
\ord_E (\env_X(K_X) + \env_X(-K_X) ) =\ord_E ( \AX - \AX^+) ,
\]
where $\AX$ is in Theorem 6.7.
However, since $X$ is valuative log canonical and 
$E$ is an exceptional divisor on $W$ ,
 we have $\ord_E ( \AX - \AX^+ ) <0 $.
Thus $\mathrm{c}_X (E) $ is the contained in $\mathrm{Nnum}$, and the theorem follows.

\end{proof}



Using Theorem 6.9,
we can reduce Theorem 1.6 to Theorem 6.5 by the same argument in proof of Theorem 1.4.

\begin{thm}
Let $X$ be a normal projective variety over $\C$, where $k$ is an algebraically closed field  and 
$f : X \to X$ be a non-invertible polarized endomorphism with \'etale in codimension one.
Then $K_X$ is $\Q$-linearly trivial.
In particular $X$ is of Calabi-Yau type.
\end{thm}

\begin{proof}
By Theorem 6.9, the non numerically $\Q$-Gorenstein locus is closed.
If this set is nonempty, we take an irreducible component $Z$.

\begin{cl}
$Z$ is totally invariant up to replacing $f$ some iterate.
\end{cl}

\begin{clproof}
This proof is very similar to the argument in the proof of Theorem 5.1.
We take $x \in X$ with $f(x)$ is the generic point of $Z$.
Then $K_X$ is not $\Q$-Cartier since $f_* K_X = K_X$ is not $\Q$-Cartier at the generic point of $Z$.
So we have
\[
f^{-1} (Z) \subset \mathrm{Nnum} (X)
\]
By the same argument in the proof of Theorem 5.1,
$Z$ is totally invariant up to replacing $f$ some iterate.
\end{clproof}

We may assume $Z$ is totally invariant of $f$, and by Theorem 6.9,
$Z$ is an also irreducible component of non $\Q$-Gorenstein varieties.
By Theorem 6.5 and similar argument of the proof of Theorem 5.1,
$K_X$ is a $\Q$-Cartier divisor on generic point of $Z$.
It is contradiction so we have $K_X$ is $\Q$-Cartier divisor.

Next, we take an ample divisor $H$ associated to $f$ and we consider an normal affine cone $\tilde{X}$ associated to $H$ and associated local endomorphism
$\tilde{f} : (\tilde{X} , 0) \to (\tilde{X} , 0) $.
Thus, $K_{\tilde{X}}$ is $\Q$-Cartier on $\tilde{X} \backslash 0$ and $\deg (\tilde{f}) \geq 2$,
therefore $K_{\tilde{X}}$ is $\Q$-Cartier by Theorem 5.1 and Proposition 6.6. 
Then $K_X \sim_{\Q} aH$ for some $a \in \Q$ by \cite[Lemma 2.32]{BdFF},
and $af^* K_X \sim_{\Q} aq K_X$ for some integer $q > 1$.
Since $f$ is \'etale in codimension one, we have $a=0$.
\end{proof}


\section{Singularities of non-$\Q$-Gorenstein pairs admitting an int-amplified endomorphism}
In this section, we consider the generalization of Theorem 1.4.
First we replace polarized endomorphisms to int-amplified endomorphisms.

\begin{defn}
Let $f : X \to X$ be a surjective endomorphim of a projective variety $X$.
We say that $f$ is \emph{int-amplified} if $f^* L - L$ is an ample Cartier divisor for some ample Cartier divisor $L$.
\end{defn}

By the following lemma, we can reduce some global question about int-amplified endomorphisms to some local question as about polarized endomorphisms.

\begin{lem}[\textup{\cite[Lemma 3.12]{meng}}]
Let $f : X \to X $ be an int-amplified endomorphism of a projective variety $X$.
Let $Z$ be an $f^{-1}$-periodic closed subvariety of $X$ such that 
$\deg (f|_Z) = \deg (f)$. Then $Z = X$.
\end{lem}

Next we consider pairs with reduced totally invariant divisor.

\begin{defn}
Let $X$ be a normal integral scheme essentially of finite type over $\C$.
Let $\Delta$ be a reduced divisor on $X$.
Then we set
\[
A^+_{\X/(X,\Delta)} := K_{\X} + 1_{\X/X} + \Delta - \env_X(K_X + \Delta)
\]
and if  $X$ has the unique closed point $0$ and $A^+_{\X/(X,\Delta)}$ is bounded below, we set
\[
\mathrm{Vol}(X,\Delta,0)^+ := -(\env_{\X}(A^+_{\X/(X,\Delta)}))^n.
\]
We remark that $A^+_{\X/(X,\Delta)}$ is exceptional.

\end{defn}

The following Theorem is generalization of Theorem 1.3

\begin{thm}[\textup{cf. Theorem 4.12}]
Let $X$ be a normal integral scheme essentially of finite type over $\C$ with unique closed point $0$.
Let $f : X \to X$ be an endomorphism with $f^{-1}(0) = \{0 \}$.
Let $\Delta$ be a reduced totally invariant divisor on $X$.
If $\tilde{d} > 1$ and $A^+_{\X/(X,\Delta)}$ is grater than or equal to $0$ on $X \backslash 0$,
then we have $A^+_{\X/(X,\Delta)} \geq 0$.
\end{thm}

\begin{proof}
By the same argument in the proof of Proposition 4.10, $A^+_{\X/(X,\Delta)}$ is bounded below.
By \cite[Lemma 2.5]{BH}, there exists an effective divisor $R_{\Delta}$ such that 
\[
K_X+\Delta  =f^*(K_X+\Delta) +R_{\Delta}
\]
holds, and we have
\[
f^*(A^+_{\X/(X,\Delta)})=K_{\X}+1_{\X/X}+\Delta +R_{\Delta}-\env_X(K_X+\Delta -R_\Delta)
\geq A^+_{\X/(X,\Delta)}
\]
by the same argument in the proof of Proposition 4.11.
Therefore we have 
\[
\tilde{d} \cdot \mathrm{Vol}(X,\Delta,0)^+ \leq \mathrm{Vol}(X,\Delta,0)^+,
\]
so $\tilde{d} > 1$ implies $\mathrm{Vol}(X,\Delta,0)^+ = 0$ and $A^+_{\X/(X,\Delta)} \geq 0$.
\end{proof}

The following Corollary is a generalization of Theorem 1.4.

\begin{cor}[\textup{cf. Theorem 5.2}]
Let $f: X \to X$ be an int-amplified endomorphism of a normal projective variety.
Let $\Delta$ be a reduced totally invariant divisor on $X$.
Then $A^+_{\X/(X,\Delta)} \geq 0$ holds.
\end{cor}

\begin{proof}
By the same argument in the proof of Theorem 5.2,
the non valuative log canonical locus of $(X,\Delta )$ is closed and totally invariant.
Suppose the set is non empty and we take an irreducible component $Z$.
We may assume $Z$ is totally invariant, and localizing at the generic point of $Z$,
we obtain $\deg (f|_Z) = \deg (f)$ by Theorem 7.4,
but it is contradiction to Lemma 7.2.
\end{proof}

Next we consider the cases where $R_\Delta$ contain an nonzero effective numerically $\Q$-Gorenstein divisor.

\begin{thm}
Let $f : X \to X$ be an endomorphism of a normal integral scheme essentially of finite type over $\C$.
Let $\Delta$ be a reduced totally invariant divisor on $X$ such that $A^+_{\X/(X,\Delta)} \geq 0$ .
Let $D$ be an effective numerically $\Q$-Cartier divisor with $D \leq R_{\Delta}$.
Then for any prime exceptional divisor $E$ over $X$ with $\ord_E (A^+_{\X/(X,\Delta)}) = 0$,
$\mathrm{c}_X(E)$ is not contained in $\Supp (D)$.
\end{thm}

\begin{proof}
By the proof of Theorem 7.4, we have
\[
f^*(A^+_{\X/(X,\Delta)})=K_{\X}+1_{\X/X}+\Delta +R_{\Delta}-\env_X(K_X+\Delta -R_\Delta)
\geq A^+_{\X/(X,\Delta)} -R_{\Delta} +\env_X(R_{\Delta}).
\]
If $E$ be an exceptional prime divisor over $X$ with $\ord_E (A^+_{\X/(X,\Delta)}) = 0$ 
and we set $f^* E = r_1 F_1 + \cdots + R_s F_s$,
we have
\[
0 = r_i \ord_E ((A^+_{\X/(X,\Delta)})) = \ord_{F_i}(f^*(A^+_{\X/(X,\Delta)})) \geq \ord_{F_i}( A^+_{\X/(X,\Delta)}).
\]
Since $A^+_{\X/(X,\Delta)} \geq 0$, we have $ \ord_{F_i}( A^+_{\X/(X,\Delta)}) =0 $,
and any valuative lc center of $(X,\Delta)$ is totally invariant replacing $f$ by some iterate.
Furthermore, since
\[
0 = r_i \ord_E ((A^+_{\X/(X,\Delta)}) = \ord_{F_i}(f^*(A^+_{\X/(X,\Delta)})) \geq \ord_{F_i}(\env_X(D)),
\]
the center of $F_i$ is not contained in $\Supp(D)$.
Here, the center of $F_i$ coincides with the center of $E$.

\end{proof}

The following proposition also follows from the same argument in \cite{BH} by replacing pullbacks by numerical pullbacks.

\begin{prop}[\textup{cf. \cite[Lemma 2.10]{BH}}]
Let $f : X \to X$ be an int-amplified endomorphim of a normal surface.
Let $\Delta$ be a reduced totally invariant divisor of $X$.
Then $(X,\Delta)$ satisfies the following conditions;
\begin{itemize}
\item $(X,\Delta)$ is a $\Q$-Gorenstein lc pair,
\item any lc center of $(X,\Delta)$ is totally invariant replacing $f$ by some iterate, and
\item any lc center of $(X,\Delta)$ is not contained in $R_{\Delta}$.
\end{itemize}
\end{prop}

\begin{proof}
First we remark that every divisors on surfaces are numerically $\Q$-Cartier.
By Corollary 7.5, $(X,\Delta)$ is valuative log canonical numerically $\Q$-Gorenstein pair,
so by the existence of lc model, we obtain the first statement of this Proposition.
By \cite[Lemma 2.10]{BH} or Theorem 7.6 , we obtain the rest of statement.
\end{proof}







\end{document}